\documentclass[11pt,reqno]{amsart}
\usepackage{amssymb,mathrsfs,color}
\usepackage{pinlabel}

\usepackage{amsmath, amsfonts, amsthm, verbatim}
\usepackage{epstopdf, mathtools}
\mathtoolsset{showonlyrefs}
\usepackage{color}
\usepackage{esint}
\usepackage{stmaryrd}

\usepackage{empheq}
\usepackage[shortlabels]{enumitem}
\usepackage{appendix}

\newcommand{\rar}{\rightarrow}

\def\grad {{\nabla}}

\DeclareMathOperator{\Sym}{Sym}

\newcommand{\bs}[1]{\boldsymbol{#1}}

\renewcommand{\div}{\operatorname{div}}

\definecolor{deepgreen}{cmyk}{1,0,1,0.5}

\newcommand{\p}{\partial}

\makeatletter

\newcommand{\Rmnum}[1]{\expandafter\@slowromancap\romannumeral #1@}
\makeatother

\newcommand{\ang}[1]{\left\langle{#1}\right\rangle}

\newcommand{\Del}[1]{}

\numberwithin{equation}{section}

\newtheorem{thm}{Theorem}[section]

\newtheorem{lem}[thm]{Lemma}

\theoremstyle{remark}

\newtheorem{defn}[thm]{Definition}

\usepackage{tensor}

\usepackage{graphicx}
\usepackage{xcolor} 
\usepackage{tensor}
\usepackage{slashed}
\usepackage{cite}
\definecolor{green}{rgb}{0,0.8,0} 

\newcommand{\diam}[1]{\overset{\diamond}{#1}}

\newcommand{\tr}{\textrm{tr}}

\newcommand{\eps}{\epsilon}

\newcommand{\bbN}{\mathbb N}

\newcommand{\bbQ}{\mathbb Q}
\newcommand{\bbR}{\mathbb R}

\newcommand{\bbT}{\mathbb T}

\newcommand{\bbZ}{\mathbb Z}

\begin{document}

\title[Two-dimensional rate-type flows with stress diffusion]{Global well-posedness for two-dimensional flows of viscoelastic rate-type fluids with stress diffusion}

\author{Miroslav Bul\'\i{}\v{c}ek, Josef M\'{a}lek, Casey Rodriguez}

\begin{abstract}
We consider the system of partial differential equations governing two-dimensional flows of a robust class of viscoelastic rate-type fluids with stress diffusion, involving a general objective derivative. 
The studied system generalizes the incompressible Navier--Stokes equations for the fluid velocity $\bs v$ and pressure $p$ by the presence of an additional term in the constitutive equation for the Cauchy stress expressed in terms of a positive definite tensor $\bs B$. The tensor $\bs B$ evolves according to a diffusive variant of an equation that can be viewed as a combination of corresponding counterparts of Oldroyd-B and Giesekus models. Considering spatially periodic problem, we prove that for arbitrary initial data and forcing in appropriate $L^2$ spaces, there exists a unique globally defined weak solution to the equations of motion, and more regular initial data and forcing launch a more regular solution with $\bs B$ positive definite everywhere.
\end{abstract}
\maketitle

\centerline{\emph{ In memory of Anton\'\i{}n Novotn\'{y}}}
\section{Introduction}

\subsection{Viscoelastic rate-type fluids with stress diffusion}

Many real fluids are not Newtonian as they cannot by accurately modeled by the standard Navier--Stokes equations. One family of models frequently encountered in the theory of such \emph{non-Newtonian fluids} is the family of incompressible viscoelastic rate-type models with stress diffusion, as these fluids are capable of describing (non-Newtonian) phenomena such as stress relaxation, nonlinear creep, normal stress differences, shear or vorticity banding. For these models, the governing equations of motion generalize the Navier--Stokes equations for the fluid velocity $\bs v$ (and the mean normal stress $p$) in the sense that the constitutive equation for the Cauchy stress contains additional terms expressed in terms of a positive definite second order tensor $\bs B$ modeling the ``elastic" energy storage mechanism of the fluid. The evolution of $\bs B$ is then described by an additional diffusive equation containing the objective time derivative of $\bs B$ having the form of a nonlinear operator involving the velocity gradient.
We refer the reader to \cite{el-kareh_leal} and \cite{MalekPrusaSkrivanSuli18} (and the references cited therein) for an introduction to the mechanics of viscoelastic rate-type fluids with stress diffusion and their applications.

In a recent study \cite{BBJ21}, the authors presented two results of apparently different nature. First of all, using the principles of continuum thermodynamics, they developed a robust class of viscoelastic-rate type models with stress diffusion that, for small elastic responses, coincide with the combination of Oldroyd-B and Giesekus models with stress diffusion, being in the form of the Laplace operator. Robustness of the developed models consists in considering a very general form of the nonlinear objective tensorial time derivative that includes the upper convected Maxwell, Jaumann--Zaremba and Gordon--Schowalter derivatives as special cases. Second, for such robust class of incompressible fluids, they established long-time and large-data existence of weak solution to three-dimensional flows in a closed bounded domains. Thus, the mathematical foundation established by Leray and his followers for the three-dimensional Navier--Stokes equations (see \cite{leray.j:sur,Hopf1950, temam.r:navier-stokes,constantin.p.foias.c:navier-stokes}) is, due to \cite{BBJ21}, available for large class of complex fluids as well.

In this paper, we focus on two-dimensional flows for the same class of fluids (i.e. those studied in \cite{BBJ21}) asking the question of existence, uniqueness and full regularity of weak solution for any regular enough data. We give the affirmative answer to this question. For brevity, we investigate the initial, spatially periodic problem. The whole problem can be then formulated in the following way.

Let $\bbT^2 = \bbR^2 / \bbZ^2$ denote a two-dimensional torus. Let $a \in \bbR$, $\beta \in (0,1)$ and $\delta_1, \delta_2 \geq 0$. Then, for $(t,x) \in [0,T] \times \bbT^2$ and for given $\bs v_0: \bbT^2 \to \bbR^2$, $\bs B_0:\bbT^2 \to \bbR^{2\times 2}$ and $\bs f:[0,T]\times \bbT^2 \to  \bbR^2$, we look for functions $\bs v$, $p$ and $\bs B$ satisfying
    \begin{align}
	&\div \bs v = 0, \label{eq:divergence}\\
	&\p_t \bs v + (\bs v \cdot \nabla) \bs v = \div \bs T +  \bs f, \label{eq:momentum}\\
	&\qquad \textrm{with } \, \bs T = -p \bs I + 2 \bs D(\bs v) + 2a[ (1 - \beta)(\bs B - \bs I) + \beta(\bs B^2 - \bs B)], \label{eq:Cauchy}\\
	&\diam{\bs B} + \delta_1 (\bs B - \bs I) + \delta_2 (\bs B^2 - \bs B) = \Delta \bs B, \label{eq:tensb}\\
	&\bs{v}(0,x) = \bs{v}_0(x), \,\, \bs{B}(0,x) = \bs{B}_0(x),\label{eq:initial}
	\end{align}
where, for any tensor $\bs A$, the symbol $\diam{\bs A}$ denotes an objective time derivative\footnote{The additional terms $- a (\bs D(\bs v)\bs B + \bs B \bs D(\bs v) ) -
(\bs W(\bs v) \bs B - \bs B \bs W(\bs v) )$ in the definition of the objective time derivative guarantee that $\diam{\bs B}$ satisfies the principle of material frame-indifference (see \cite{NLFTM} for a discussion of this principle in mechanics).} of the symmetric, positive definite second order tensor field $\bs A$ given by
\begin{align}
	\diam{\bs A} = \p_t \bs A + (\bs v \cdot \nabla) \bs A - a (\bs D(\bs v)\bs A + \bs A \bs D(\bs v) ) -
	(\bs W(\bs v) \bs A - \bs A \bs W(\bs v) ). \label{eq:obj_derivative}
\end{align}
The symmetric and skew parts $\bs D(\bs v)$ and $\bs W(\bs v)$ of the velocity gradient $\grad \bs v$ are defined as
\begin{align*}
\bs D(\bs v) = \frac{1}{2}[(\grad \bs v) + (\grad \bs v)^T] \quad \textrm{ and } \quad \bs W(\bs v) = \frac{1}{2}[(\grad \bs v) - (\grad \bs v)^T].
\end{align*}
Equation \eqref{eq:divergence} expresses the incompressiblity constraint, equation \eqref{eq:momentum} expresses the balance of linear momentum where the constant density is set to be one for simplicity, equation \eqref{eq:Cauchy} is the constitutive equation for the Cauchy stress $\bs T$ where, again for simplicity, we set the viscosity to be one, and equation \eqref{eq:tensb} describes the evolution of the ``elastic'' part of the overall deformation of the fluid (see \cite{RAJAGOPAL2000} and \cite{MalekPrusaChapter} for more details providing clear physical interpretation of $\bs B$ associated with the concept of evolving natural configuration). The last equation \eqref{eq:initial}  characterizes the intial state of the system.

As usual, for spatially periodic problem associated with the equations of the Navier--Stokes type (see \cite{Temam85}), we assume that
$$
\bs v_0, \, \bs f(t, \cdot), \, \bs v (t, \cdot) \textrm{ and } p(t, \cdot) \textrm{ have zero mean value over } \bbT^2.
$$

When $\bs B = \bs I$, then the governing equations reduce to the standard incompressible Navier--Stokes equations. Note that the Navier--Stokes equations are decoupled from the equation \eqref{eq:tensb} for $\bs B$ if we set $a = 0$ in the formula for the Cauchy stress \eqref{eq:Cauchy}. Although usually the parameter $a$ does not occur in the formula for the Cauchy stress, Larson \cite{Larson88} (see pages 131--133) provided arguments why the parameter $a$ should be there. The analysis presented in \cite{BBJ21} and in this study holds for all $a\in\bbR$. If the reader would prefer to consider the model without the presence of $a$ in the form for the Cauchy stress, then the results proved in \cite{BBJ21} and herein hold for $a>0$. It is also worth mentioning that the limiting cases of the model we consider include diffusive variants of the following:
\begin{itemize}
	\item Oldroyd-B model \cite{Oldroyd} ($a = 1$, $\beta = 0$, $\delta_1 > 0$, $\delta_2 = 0$),
	\item Giesekus model \cite{GIESEKUS1982} ($a = 1$, $\beta = 0$, $\delta_1 = 0$, $\delta_2 > 0$),
	\item Johnson-Segalman model \cite{JOHNSON1977} ($a \in [-1,1]$, $\beta = 0$).
\end{itemize}
The results proved in \cite{BBJ21} and herein require that $\beta\in (0,1)$. This essential assumption is linked with a modification of the constitutive equation for the Helmholtz free energy (see \eqref{eq:helm} below). For small elastic strains (i.e. when $|\bs B - \bs I|$ is small) the elastic responses of classical models (with $\beta =0$) and the elastic responses of the models considered here
coincide, see \cite{BBJ21} for details.

Note that the diffusive Johnson-Segalman model (and consequently also its generalization) is perceived to be the model
of the \emph{shear-banding phenomenon}, which is a phenomenon of eminent interest in the applications; see for example \cite{OlmstedRadulescu2000,olmsted.pd:perspectives,fardin.m.radulescu.o.ea:stress,divoux.t.fardin.ma.ea:shear}.

A basic and essential feature for the specification of the concept of weak solution is the energy balance for the system of governing equations that are being considered. In our case, for sufficiently regular solutions to \eqref{eq:divergence}--\eqref{eq:initial}, the energy balance takes the form  \begin{align}
	\frac{d}{dt} \int_{\bbT^2} \frac{1}{2}|\bs v|^2 dx + \frac{d}{dt} \int_{\bbT^2} \psi(\bs B)\, dx =
	-\int_{\bbT^2} \xi \, dx + \int_{\bbT^2} \bs f \cdot \bs v \, dx \label{eq:energybalance}
\end{align}
where the \emph{Helmholtz free energy} $\psi$ is
\begin{align}
	\psi(\bs B) = (1-\beta)(\tr \bs B - 2 - \log \det \bs B) + \frac{\beta}{{2}} |\bs B - \bs I|^2 \label{eq:helm}
\end{align}
and non-negative \emph{dissipation rate} $\xi$ is
\begin{align}
 \xi &=   2|\bs D(\bs v)|^2 + (1-\beta) | \bs B^{-1/2} \nabla \bs B \bs B^{-1/2} |^2
 + \beta | \nabla \bs B |^2  \\
 &\quad+ (1-\beta) \delta_1 | \bs B^{1/2} - \bs B^{-1/2} |^2
 +
  \beta \delta_2 | \bs B^{3/2} - \bs B^{1/2} |^2 \\
 &\quad +  (\beta \delta_1 + (1-\beta) \delta_2) | \bs B - \bs I |^2. \label{eq:diss}
\end{align}
This can be seen by summing \eqref{eq:momentum} scalarly multiplied by $\bs v$ and \eqref{eq:tensb} scalarly multiplied by $\bs J = \p_{\bs B} \psi(\bs B) = (1-\beta)(1 - \bs B^{-1}) + \beta (\bs B - \bs I)$ followed by the integration over $\bbT^2$ and integration by parts, using \eqref{eq:divergence}. See \cite{BBJ21} for details. Note that for the Navier--Stokes equations, when $\bs B = \bs I$, $\psi = 0$ and $\xi$ contains only the first term in \eqref{eq:diss}. For $\bs B\neq\bs I$, the structure of energy storing and dissipative mechanisms is obviously very complex and nontrivial. Also, by comparing the models with $\beta = 0$ to our model with $\beta \in (0,1)$, we observe that the models with $\beta \in (0,1)$ has an attractive mathematical feature that the total energy $\int_{\bbT^2} \frac{1}{2} |\bs v|^2 dx + \int_{\bbT^2} \psi(\bs B) dx$ and dissipation rate $\xi$ are coercive in norms of $\bs v$ and $\bs B$ of \emph{equal} regularity.

In fact, there is a thermodynamical approach that specifies the form of the Cauchy stress tensor $\bs T$ (i.e a \emph{tensorial} quantity) from the knowledge of the constitutive equations for two \emph{scalars} $\psi$ and $\xi$  (such as \eqref{eq:helm} and \eqref{eq:diss} above). For more details regarding this process of determining the form of $\bs T$ for the fluids based on positing the mechanisms of energy storage and dissipation (two scalar functions), we refer the reader to \cite{RAJAGOPAL2000}, \cite{RajSri2004}, \cite{MRT2015}, \cite{MalekPrusaChapter} and \cite{MalekPrusaSkrivanSuli18}.

\subsection{Main results and outline}

In \cite{BBJ21}, there was established the long-time existence of weak solutions (for large initial data) to \eqref{eq:divergence}--\eqref{eq:initial} considered on three-dimensional Lipschitz domains for Navier slip boundary conditions.  A major difficulty in using the energy balance \eqref{eq:energybalance} in an approximation scheme is justifying that $\bs B^{-1}$ exists almost everywhere and is in an appropriate test space. The work \cite{BBJ21} introduced a novel two-step approximation scheme for $(\bs v, \bs B)$ with parameters $(k, \eps) \in \bbN \times (0,1)$ involving the minimal eigenvalue $\Lambda(\bs B)$ of $\bs B$ so that:
\begin{itemize}
	\item in step one, a uniform-in-$k$ energy estimate is obtained by pairing a certain Galerkin system only with the approximates $(\bs v^k_{\eps}, \bs B^k_{\eps})$,
	\item in step two, the weak limit $(\bs v_{\eps}, \bs B_{\eps})$ is shown to satisfy $\Lambda(\bs B_{\eps}) \geq \eps$ almost everywhere allowing for an approximate form of \eqref{eq:energybalance} (uniform-in-$\eps$) to be obtained.
\end{itemize}

In this work we push this methodology further and establish the global existence and uniqueness of weak solutions to \eqref{eq:divergence}--\eqref{eq:initial} for large initial data in two dimensions. We work on $\bbT^2$ for simplicity, but it is likely one can incorporate boundaries and different boundary conditions consistent with the second law of thermodynamics. In Section 2 we give the precise definition of weak solutions we study (see Definition \ref{d:weaksoln}) and statements of our main results: global existence and uniqueness of weak solutions (see Theorem \ref{t:main2}) and higher regularity of weak solutions (see Theorem \ref{t:main1}). In Section 3 we prove uniqueness of weak solutions using a combination of H\"older's inequality, Sobolev embedding and interpolation.  An adaptation of the Galerkin approximation scheme introduced in \cite{BBJ21}, where the existence of weak solutions was established for large data, is given in the appendix. Finally, in Section 4 we establish simple propagation of regularity results (see Lemma \ref{l:reg} and Lemma \ref{l:higherregu}) and prove that smooth initial data and forcing launch smooth solutions with $\bs B$ positive definite everywhere, rather than only almost everywhere (see Proposition \ref{t:main1}), proving Theorem~\ref{t:main1}.

We comment that the proof of our main result can be contrasted with the work \cite{ConstKliegl2dOldroyd} on the standard Oldroyd-B model in two dimensions. There, Constantin and Kliegl establish global well-posedness for strong solutions by using the maximum principle to verify $\bs B$ is positive definite everywhere and using \eqref{eq:energybalance} to obtain sufficient apriori bounds. In contrast, for our model we prove that a global weak solution exists, is unique, and a posteriori, we show that more regular data launches a more regular solution with $\bs B$ positive definite everywhere. Finally, we remark that our main results along with \cite{BBJ21} effectively prove that solutions to \eqref{eq:divergence}--\eqref{eq:initial} enjoy at least the same level of regularity as solutions to the standard incompressible Navier--Stokes equations in dimensions two and three.

\section{Weak solutions}

\subsection{Notation}

In what follows we denote scalar, vector and tensor quantities by $a$, $\bs a$ and $\bs A$ respectively. We use the same notation for function spaces of scalar, vector or tensor-valued functions, but the context should be clear. The Lebesgue and $L^2$-based Sobolev spaces on $\bbT^2$ are denoted by $L^p$ and $H^s$, respectively. The $L^2$ pairing and norm are denoted by $(\cdot,\cdot)$ and $\| \cdot \|$ respectively, and the pairing between dual spaces (clear from the context) is denoted by $\ang{\cdot, \cdot}$. We denote
\begin{align*}
L^2_0 = \Bigl \{ \bs v \in L^2 : \int_{\bbT^2} \bs v \, dx = \bs 0 \Bigr \},
\end{align*}
and for $s \geq 0$,
\begin{align*}
H^s_{0,\div} &= \Bigl \{ \bs v \in H^s \cap L^2_{0} : \div \bs v = 0  \Bigr \}, \\
H^{-s}_{0,\div} &= \Bigl (H^s_{0,\div}\Bigr)^*.
\end{align*}
The second condition in the definition of $H^s_{0,\div}$ is interpreted in the sense of distributions. In addition, we denote the set of  $2 \times 2$ real symmetric matrices by $\Sym$ and the set of  $2 \times 2$ real symmetric positive definite matrices by $\Sym_+$. Moreover, we define
\begin{align*}
H^s_{\Sym} &= \left \{ \bs A \in H^s  : \bs A \in \Sym \right\}, \\
L^{p}_{\Sym} &= \left \{ \bs A \in L^p  : \bs A \in \Sym \right\}.
\end{align*}

For $\bs A \in \bbR^{2 \times 2}$, we denote
\begin{align}
	\bs S(\bs A) &= (1-\beta) (\bs A - \bs I) + \beta (\bs A^2 - \bs A), \\
	\bs R(\bs A) &= \delta_1(\bs A - \bs I) + \delta_2(\bs A^2 - \bs A).
\end{align}
We note that the Cauchy stress tensor can then be written as
\begin{align}
	\bs T = -p \bs I + 2 \bs D(\bs v) + 2a \bs S(\bs B).
\end{align}

\subsection{The definition of weak solutions and statement of the main result}

\begin{defn}\label{d:weaksoln}
Let $T > 0$. Let $\beta \in (0,1), \delta_1, \delta_2 \geq 0$, $a \in \bbR$. Let $\bs f \in L^2(0,T; H^{-1}_{0,\div})$.
A pair $(\bs v, \bs B) : [0,T] \rar \bbR^2 \times \mbox{Sym}_+$ is a weak solution to \eqref{eq:divergence}--\eqref{eq:initial} if the following hold:
\begin{enumerate}[(a)]
	\item $\bs v \in L^2(0,T;H^1_{0,\div}) \cap C([0,T];L^2)$ and $\p_t \bs v \in L^2(0,T; H^{-1}_{0,\div}),$
	
	\item $\bs B \in L^2(0,T;H^1_{\Sym}) \cap C([0,T];L^2)$ and \newline
	$\p_t \bs B \in [L^2(0,T; H^1) \cap L^4(0,T;L^4)]^*,$
	
	\item  the following two identities hold:
	\begin{equation}
	\begin{aligned}
		\int_0^T &[\ang{\p_t \bs v, \bs \varphi} + ((\bs v \cdot \nabla) \bs v, \bs \varphi) ] dt \\
		&= -\int_0^T (\nabla \bs v, \nabla \bs \varphi) - 2a (\bs S(\bs B), \grad \bs \varphi) + \ang{\bs f, \bs \varphi} dt, \label{eq:veq}\\
		&\textrm{for all $\bs \varphi \in L^2(0,T; H^1_{0,\div})$},
	\end{aligned}
	\end{equation}
	\begin{equation}
	\begin{aligned}
	\int_0^T &[\ang{\p_t \bs B, \bs A } + ((\bs v \cdot \nabla ) \bs B + 2 \bs B\bs W(\bs v) - 2 a \bs B \bs D(\bs v), \bs A) ] dt \\
	&+ \int_0^T (\bs R (\bs B), \bs A) dt =
	- \int_0^T (\nabla \bs B, \nabla \bs A) dt. \label{eq:Beq}\\
	&\textrm{for all  $\bs A \in L^2(0,T; H^1_{\Sym}) \cap L^4(0,T;L^4)$}
	\end{aligned}
\end{equation}
\end{enumerate}
\end{defn}

A few remarks in order. The pressure $p$ does not appear in Definition \ref{d:weaksoln} since we are essentially considering weak solutions to
\begin{align}
	\p_t \bs v + P((\bs v \cdot \nabla)\bs v) &= \Delta \bs v + 2a P(\div \bs S(\bs B)) + \bs f, \\
	\diam{\bs B} + \bs R(\bs B) &= \Delta \bs B, \label{eq:Leray3}
\end{align}
where $P : L^2_0 \rar L^2_{0,\div}$ is the Leray projector. However, if the external forces~$\bs f$ are slightly more regular, then the pressure can be defined a~posteriori (see Lemma \ref{l:reg} or proof of Theorem~\ref{t:main2}).

We also recall the Ladyzhenskaya inequality valid in two dimensions
\begin{equation}
    \label{Lady}
    \|u\|^2_4 \le C\|u\| \|u\|_{H^1}.
\end{equation}
Then it is also remarkable that  by H\"older's inequality, it follows that $\bs B$ from Definition \ref{d:weaksoln} satisfies
\begin{align*}
	\int_0^T \| \bs B \|^4_{L^4} dt \leq C\| \bs B \|_{L^\infty L^2}^2
	\int_0^T \| \nabla \bs B \|^2+\|\bs B\|^2 dt.
\end{align*}
Thus, $\bs B$ can be paired with $\p_t \bs B$. By an application of the Hahn-Banach theorem, we have
\begin{align*}
[L^2(0,T; H^1) \cap L^4(0,T;L^4)]^* =
L^2(0,T;H^{-1}) + L^{4/3}(0,T;L^{4/3}),
\end{align*}
 and a mollification (in time) argument yields $\bs B \in C([0,T];L^2)$, $t \mapsto \| \bs B(t)\|^2$ is absolutely continuous and
\begin{align*}
	\frac{d}{dt} \| \bs B(t) \|^2 = 2 \ang{\p_t \bs B(t), \bs B(t)}
\end{align*}
a.e. on $[0,T]$; see e.g. \cite{EvansPDE} for the standard argument involving functions in $L^2(0,T;H^1)$ with time derivatives in $L^2(0,T;H^{-1})$.

The main results of this work are the following existence and uniqueness of global weak solutions, and the existence and uniqueness of global smooth solutions for smooth initial data and forcing.

\begin{thm}\label{t:main2}
Let $T > 0$. Let $\beta \in (0,1), \delta_1, \delta_2 \geq 0$, $a \in \bbR$. Let
\begin{align}
	&\bs f \in L^2(0,T; H^{-1}_{0,\div}), \quad \bs v_0 \in L^2_{0,\div}, \\
	&\bs B_0 : \bbT^2 \rar \mbox{Sym}_+, \bs B_0 \in L^2 \textrm{ and } \log \det \bs B_0 \in L^1.
\end{align}
Then there exists a unique weak solution to \eqref{eq:divergence}--\eqref{eq:initial} on $[0,T] \times \bbT^2$ satisfying
$\bs v(0) = \bs v_0$, $\bs B(0) = \bs B_0$.
\end{thm}

\begin{thm}\label{t:main1}
	Let $\beta \in (0,1), \delta_1, \delta_2 \geq 0$, $a \in \bbR$.
	Suppose that $\bs v_0$ is a smooth vector field on $\bbT^2$ with $\int_{\bbT^2} \bs v_0 \, dx = \bs 0$, $\bs B_0$ is a smooth, symmetric, positive definite tensor field on $\bbT^2$, and $\bs f$ is a smooth vector field on $[0,\infty) \times \bbT^2$ with $\div \bs f = 0$. Then there exist a unique smooth vector field $\bs v$ with $\int_{\bbT^2} \bs v \, dx = \bs 0$, a unique smooth scalar function $p$ with $\int_{\bbT^2} p \, dx = 0$, and a unique symmetric, positive definite tensor field $\bs B$ on $[0,\infty) \times \bbT^2$ solving \eqref{eq:divergence}--\eqref{eq:initial}.
\end{thm}

\section{Uniqueness of Weak Solutions}
We first prove the uniqueness of weak solutions as defined in Section 2. The proof of existence largely follows as in the three-dimensional case \cite{BBJ21} and can be found in the appendix of this work.

\subsection{Proof of uniqueness of weak solutions}\label{s:uniqueness}
Let $(\bs v, \bs B)$ and $(\bs u, \bs A)$ be two weak solutions to \eqref{eq:divergence}--\eqref{eq:initial} with initial data $(\bs v_0, \bs B_0)$. Define $\bs w = \bs v - \bs u$  and $\bs C = \bs B - \bs A$. Then $(\bs w, \bs C)$ satisfy, along with the initial condition $\bs w(0) = \bs 0$, $\bs C(0) = \bs 0$ two weak identities:
\begin{equation}
\begin{aligned}
	\int_0^T &[\ang{\p_s \bs w, \bs \varphi} + (\nabla \bs w, \nabla \bs \varphi) ] ds \\
	= &-\int_0^T ((\bs v \cdot \nabla) \bs w + (\bs w \cdot \nabla) \bs u, \bs \varphi) ds \\
	&- 2a \int_0^T (\bs S (\bs B)- \bs S (\bs A), \grad \bs \varphi) ds\\
	\label{eq:weq}
	&\textrm{for all $\bs \varphi \in L^2(0,T; H^1_{0,\div})$}
\end{aligned}
\end{equation}
and
\begin{equation}\label{eq:Ceq}
\begin{aligned}
\int_0^T &[ \ang{\p_s \bs C, \bs E} + (\nabla \bs C, \nabla \bs E) ] ds \\
=&-\int_0^T ((\bs v \cdot \nabla) \bs C + (\bs w \cdot \nabla)\bs A, \bs E) ds  \\
&+ a \int_0^T (\bs D(\bs v) \bs C + \bs C \bs D(\bs v) + \bs D(\bs w) \bs A + \bs A \bs D(\bs w)), \bs E) ds  \\
&+ \int_0^T (\bs W(\bs v) \bs C - \bs C \bs W(\bs v) + \bs W(\bs w) \bs A - \bs A \bs W(\bs w)), \bs E) ds \\
&- \int_0^T  (\bs R(\bs B) - \bs R(\bs A), \bs E) ds\\
&\textrm{for all $\bs E \in L^2(0,T; H^1_{\Sym}) \cap L^4(0,T;L^4)$.}
\end{aligned}
\end{equation}

Let $t \in [0,T]$. We use $\bs \varphi = \chi_{[0,t]} \bs w$ as a test function in \eqref{eq:weq} and $\bs E=\chi_{[0,t]}
\bs C$ as a test function in \eqref{eq:Ceq}  and estimate the individual terms. Since $\div \bs v = 0$, we conclude via integration by parts that
\begin{align*}
	((\bs v \cdot \nabla) \bs w, \bs w) &= \frac{1}{2} \int_{\bbT^2} \bs v \cdot \nabla |\bs w|^2 dx = 0,\\
	((\bs v \cdot \nabla) \bs C, \bs C)&=\frac{1}{2} \int_{\bbT^2} \bs v \cdot \nabla|\bs C|^2 dx = 0.
\end{align*}
Let $\epsilon > 0$, to be specified. We focus on the remaining terms on the right hand side of~\eqref{eq:weq}. Using the point-wise estimate
$$
|\bs S(\bs B) - \bs S(\bs A)| \le C(1+|\bs A|+ |\bs B|)|\bs B- \bs A|,
$$
we conclude with the help of the H\"{o}lder inequality and the Ladyzhenskaya inequality \eqref{Lady} that
\begin{align*}
     &\left|-(\bs w \cdot \nabla) \bs u, \bs w) - 2a (\bs S (\bs B)- \bs S (\bs A), \grad \bs w)\right|\\
     &\le C\int_{\bbT^2}|\bs w|^2|\nabla \bs u| + (1+|\bs A|+ |\bs B|)|\bs C||\nabla \bs w|\\
     &\le C(\|\nabla \bs u\| \|\bs w\|_4^2 +(1+\|\bs A\|_4+ \|\bs B\|_4)\|\bs C\|_4\|\nabla \bs w\| \\
     &\le C(\|\nabla \bs u\| \|\bs w\|\|\nabla \bs w\| +(1+\|\bs A\|_4+ \|\bs B\|_4)\|\bs C\|^{\frac12}(\|\bs C\|^{\frac12}+\|\nabla \bs C\|^{\frac12})\|\nabla \bs w\|\\
     &\le \epsilon (\|\nabla \bs w\|^2 +\|\nabla \bs C\|^2) + C ( \| \bs w \|^2 + \| \bs C \|^2)
	( \| \bs u \|^2 + 1 + \| \bs B \|^4_4 +\| \bs A \|^4_4).
\end{align*}
Thus, as long as $\epsilon$ is chosen sufficiently small, we can apply the Young inequality and deduce from \eqref{eq:weq} that  for all $t \in [0,T]$
\begin{align}
	&\| \bs w(t) \|^2 + \frac{199}{100}  \int_0^t \| \nabla \bs w \|^2 ds \\
	&\leq C \int_0^t ( \| \bs w \|^2 + \| \bs C \|^2)
	( \| \bs u \|^2 + 1 + \| \bs B \|^4_4 +\| \bs A \|^4_4
	) ds \\
	&+ \frac{1}{100} \int_0^t \| \nabla \bs C \|^2 ds. \label{eq:west}
\end{align}
where $C$ is a constant depending on $\beta$ and $a$.

Similarly, as we have
$$
|\bs R(\bs B)-\bs R(\bs A)|\le C(1+|\bs A|+|\bs B|)|\bs B-\bs A|
$$
we can use the H\"{o}lder inequality and the Ladyzhenskaya inequality to get for the terms on the right hand side of \eqref{eq:Ceq}
\begin{align*}
-& (\bs w \cdot \nabla)\bs A, \bs C) + a (\bs D(\bs v) \bs C + \bs C \bs D(\bs v) + \bs D(\bs w) \bs A + \bs A \bs D(\bs w)), \bs C) \\
&+ (\bs W(\bs v) \bs C - \bs C \bs W(\bs v) + \bs W(\bs w) \bs A - \bs A \bs W(\bs w)), \bs C)  -  (\bs R(\bs B) - \bs R(\bs A), \bs C) \\
&\le C\int_{\bbT^2} |\bs w||\nabla \bs A||\bs C| +(1+|\bs A|+|\bs B|)|\nabla \bs v||\bs C|^2 +|\nabla \bs w||\bs C|(1+|\bs A|+|\bs B|)\\
&\le C \|\bs w\|_4 \|\nabla \bs A\|\|\bs C\|_4 +C(1+\|\bs A\|_4+\|\bs B\|_4)\|\nabla \bs v\|\|\bs C\|^2_4 +\\
&\quad +C\|\nabla \bs w\|\|\bs C\|_4(1+\|\bs A\|_4+\|\bs B\|_4)\\
&\le C \|\bs w\|^{\frac12}\|\nabla \bs w\|^{\frac12} \|\nabla \bs A\|\|\bs C\|^{\frac12} (\|\bs C\|^{\frac12}+\|\nabla \bs C\|^{\frac12})\\
&\quad+C(1+\|\bs A\|_4+\|\bs B\|_4)\|\nabla \bs v\|\|\bs C\|(\|\bs C\|+\|\nabla\bs C\|) +\\
&\quad +C\|\nabla \bs w\|\|\bs C\|^{\frac12} (\|\bs C\|^{\frac12}+\|\nabla \bs C\|^{\frac12})(1+\|\bs A\|_4+\|\bs B\|_4)\\
&\le \epsilon (\|\nabla \bs w\|^2 +\|\nabla \bs C\|^2)\\
&\quad + C(\|\bs w\|^2 + \|\bs C\|^2)((1+\|\bs A\|^4_4+\|\bs B\|^4_4+\|\nabla \bs v\|^2 + \|\nabla \bs A\|^2)
\end{align*}
%
Let $\epsilon > 0$, again to be specified.
Choosing $\eps$ sufficiently small and combining the previous estimates we conclude that for all $t \in [0,T]$,
\begin{align}
	\|\bs C(t) \|^2& + \frac{199}{100}  \int_0^t \| \nabla \bs C \|^2 ds \\
	&\leq C \int_0^t (\|\bs w\|^2 + \|\bs C\|^2)(1+\|\bs A\|^4_4+\|\bs B\|^4_4+\|\nabla \bs v\|^2 \\
	&\qquad + \|\nabla \bs A\|^2) ds \\
	&+ \frac{1}{100} \int_0^t \| \nabla \bs w \|^2 ds. \label{eq:Cest}
\end{align}
Adding \eqref{eq:west} and \eqref{eq:Cest} yields, for all $t \in [0,T]$,
\begin{align}
	\| \bs w(t) \|^2 &+ \| \bs C(t) \|^2 + \int_0^t ( \| \nabla \bs w \|^2 + \| \nabla \bs C \|^2 ) ds \\
	&\leq C \int_0^t ( \| \bs w(s) \|^2 + \| \bs C(s) \|^2) g(s) ds, \label{eq:gronwall}
\end{align}
where
\begin{align}
	g&=1+\|\bs u\|^2 + \|\bs v\|^2 + \| \nabla \bs u \|^2 +\|\nabla \bs v\|^2 + \| \bs B \|^4_4 \\&\quad +\|\bs A \|_4^4 + \|\nabla \bs B \|^2 +
	\| \nabla \bs A \|^2.
\end{align}
Now, $g \in L^1([0,T])$ by our assumptions for $(\bs v, \bs B)$ and $(\bs u, \bs A)$. By Gronwall's inequality we conclude that $(\bs w, \bs C) = (\bs 0, \bs 0)$ as desired.

\qed

\section{Higher Regularity}

In this section we prove simple propagation of regularity results for weak solutions and conclude that smooth initial data and forcing launch smooth globally defined solutions.

\subsection{Propagation of regularity}

\begin{lem}\label{l:reg}
	Let $T > 0$. Let $\beta \in (0,1), \delta_1, \delta_2 \geq 0$, $a \in \bbR$,
	\begin{align}
		\bs f \in L^2(0,T; L^2_{0,\div}), \quad \bs v_0 \in H^1_{0,\div},
	\end{align}
	and $\bs B_0 : \bbT^2 \rar \mbox{Sym}_+$ such that $\log \det \bs B_0 \in L^1$,  $\bs B_0 \in H^1$. Let $(\bs v, \bs B)$ be the unique weak solution to \eqref{eq:divergence}--\eqref{eq:initial} on $[0,T] \times \bbT^2$ such that
	\begin{align}
		\bs v(0) = \bs v_0, \quad \bs B(0) = \bs B_0.
	\end{align}
	Then
	\begin{align}
		\begin{split}
		\bs v &\in L^\infty(0,T; H^1) \cap L^2(0,T; H^2) \\
		\p_t \bs v &\in L^2(0,T; L^2_{0,\div}), \\
		\bs B &\in L^\infty(0,T; H^1) \cap L^2(0,T; H^2) \\
		\p_t \bs B &\in L^2(0,T; L^2),
		\end{split}\label{eq:dxregularity},
	\end{align}
there exists a unique scalar function $p \in L^2(0,T; H^1)$ with $\int_{\bbT^2} p \, dx = 0$ such that
\begin{align}
	\p_t \bs v + (\bs v \cdot \nabla)\bs v &= - \grad p + \Delta \bs v + 2a \div \bs S(\bs B) + \bs f, \\
	\diam{\bs B} + \bs R(\bs B) &= \Delta \bs B, \label{eq:pressure}
\end{align}
a.e. on $[0,T] \times \bbT^2$, and
	\begin{align}
		\| \bs v& \|_{L^\infty H^1} + \| \bs B \|_{L^\infty L^2} +
		\| \bs v \|_{L^2 H^2} + \| \bs B \|_{L^2 H^2}
		+ \| \p_t \bs v \|_{L^2 L^2} + \| \p_t \bs B \|_{L^2 L^2} \\
		&\leq C_1 \bigl (\| \bs v_0 \|_{H^1}, \| \psi(\bs B_0) \|_{L^1},
		\| \bs B_0 \|_{H^1}, \| \bs f \|_{L^2 L^2} \bigr ) \label{eq:dxvBenergyest}.
	\end{align}

If, in addition,
\begin{align}
 \bs v_0 \in H^2_{0,\div}, \quad \bs B \in H^2, \quad \p_t \bs f \in L^2(0,T; L^2_{0,\div}), \label{eq:H2regularity}
\end{align}
then
\begin{align}
	\begin{split}
		\bs v &\in L^\infty(0,T; H^2) \\
		\p_t \bs v &\in L^\infty(0,T; L^2_{0,\div}) \cap L^2(0,T;H^1), \\
		\p_t^2 \bs v &\in L^2(0,T; H^{-1}_{0,\div}), \\
		\bs B &\in L^\infty(0,T; H^2),\\
		\p_t \bs B &\in L^\infty(0,T; L^2) \cap L^2(0,T;H^1), \\
				\p_t^2 \bs B &\in L^2(0,T; H^{-1}),
	\end{split}\label{eq:dtregularity},
\end{align}
and
\begin{align}
	\| \bs v &\|_{L^\infty H^2} + \| \bs B \|_{L^\infty H^2} +
	\| \p_t \bs v \|_{L^\infty L^2} + \| \p_t \bs B \|_{L^\infty L^2}
	\\ &+ \| \p_t \bs v \|_{L^2 H^1} + \| \p_t \bs B \|_{L^2 H^1} +
	\| \p_t^2 \bs v \|_{L^2 H^{-1}} + \| \p_t^2 \bs B \|_{L^2 H^{-1}} \\
	&\leq C_1' \bigl (\| \bs v_0 \|_{H^2}, \| \psi(\bs B_0) \|_{L^1},
	\| \bs B_0 \|_{H^2}, \| \bs f \|_{H^1 L^2} \bigr ) \label{eq:dtvBenergyest}.
\end{align}
\end{lem}

\begin{proof}
We will give the arguments for \eqref{eq:dxvBenergyest} and \eqref{eq:dtvBenergyest} assuming the regularity \eqref{eq:dxregularity} and \eqref{eq:dtregularity} (respectively) by taking partial derivatives of the equations of motion. The proof assuming only the stated regularity of the initial data then follows from these arguments by using suitable difference quotients rather than partial derivatives.

Restating \eqref{eq:vBenergyest} we have
\begin{align}
	\| \bs v \|_{L^\infty L^2} &+ \| \bs B \|_{L^\infty L^2} +
	\| \nabla \bs v \|_{L^2 L^2} + \| \nabla \bs B \|_{L^2 L^2} \\
	&\leq C_0 \bigl (\| \bs v_0 \|, \| \psi(\bs B_0) \|_{L^1}, \| \bs f \|_{L^2 H^{-1}_{0,\div}} \bigr ) \label{eq:0energyest}.
\end{align}
Let $i \in \{1,2\}$, $\bs w = \p_{x^i} \bs v$ and $\bs C = \p_{x^i} \bs B$. Then
for all $\bs \varphi \in L^2(0,T; H^1_{0,\div})$
\begin{align}
	\int_0^T &[\ang{\p_s \bs w, \bs \varphi} + (\nabla \bs w, \nabla \bs \varphi) ] ds \\
	= &-\int_0^T ((\bs v \cdot \nabla) \bs w + (\bs w \cdot \nabla) \bs v, \bs \varphi) ds - \int_0^T (\bs f, \p_{x^i} \bs \varphi) ds\\
	&- 2a \int_0^T (
	(1 - \beta) \bs C
	+ \beta (\bs C \bs B + \bs B \bs C - \bs C, \grad \bs \varphi) ds,
	\label{eq:dxveq}
\end{align}
and for all $\bs E \in [L^2(0,T;H^1) \cap L^4(0,T;L^4)]^*$, $\bs E^T = \bs E$, we have
\begin{align}
	\int_0^T &[ \ang{\p_s \bs C, \bs E} + (\nabla \bs C, \nabla \bs E) ] ds \\
	=&-\int_0^T ((\bs v \cdot \nabla) \bs C + (\bs w \cdot \nabla)\bs B, \bs E) ds  \\
	&+ a \int_0^T (\bs D(\bs w) \bs B + \bs B \bs D(\bs w) + \bs D(\bs v) \bs C + \bs C \bs D(\bs v)), \bs E) ds \\
	&+ \int_0^T (\bs W(\bs v) \bs C - \bs C \bs W(\bs v) + \bs W(\bs w) \bs B - \bs B \bs W(\bs w)), \bs E) ds  \\
	&- \int_0^T  (\delta_1 \bs C + \delta_2 (\bs C \bs B + \bs B \bs C - \bs C), \bs E) ds. \label{eq:dxBeq}
\end{align}
We now argue identically as in Section \ref{s:uniqueness} and obtain the estimate, for all $t \in [0,T]$,
\begin{align}
\| \bs w(t) \|^2 &+ \| \bs C(t) \|^2 + \int_0^t (\| \nabla \bs w \|^2 + \| \nabla \bs C \|^2 ) ds \\
&\leq \| \bs w(0) \|^2 + \| \bs C(0) \|^2 + \int_0^t \| \bs f \| \| \bs w \| ds \\
&+ C \int_0^t (\| \bs w(s) \|^2 + \| \bs C(s) \|^2 ) g(s) ds \\
&\leq \| \bs v_0 \|_{H^1}^2 + \| \bs B_0 \|^2_{H^1} + \int_0^t \| \bs f \| \| \bs w \| ds \\
&+ C \int_0^t (\| \bs w(s) \|^2 + \| \bs C(s) \|^2 ) g(s) ds,
\end{align}
where
\begin{align}
	g = 1 + \| \nabla \bs v \|^2 + \| \nabla \bs B \|^2 + \| \bs v  \|^2 \| \nabla \bs v \|^2 + \| \bs B \|^2 \| \nabla \bs B \|^2.
\end{align}
By Gronwall's inequality and \eqref{eq:0energyest} we conclude that
\begin{align}
	\| \bs w \|_{L^\infty L^2} &+ \| \bs C \|_{L^\infty L^2} +
\| \nabla \bs w \|_{L^2 L^2} + \| \nabla \bs C \|_{L^2 L^2} \\
&\leq C_{1/3} \bigl (\| \bs v_0 \|_{H^1}, \| \psi(\bs B_0) \|_{L^1}, \| \bs B_0 \|_{H^1}, \| \bs f \|_{L^2 L^2} \bigr ), \label{eq:wCest}
\end{align}
and thus,
\begin{align}
	\| \bs v \|_{L^\infty H^1} &+ \| \bs B \|_{L^\infty H^1} +
\| \bs v \|_{L^2 H^2} + \| \bs B \|_{L^2 H^2} \\
&\leq C_{2/3} \bigl (\| \bs v_0 \|, \| \psi(\bs B_0) \|_{L^1}, \| \bs B_0 \|_{H^1}, \| \bs f \|_{L^2 L^2} \bigr ). \label{eq:dxvBest}
\end{align}
The estimate for $\p_t \bs v$ and $\p_t \bs B$ then follows from \eqref{eq:veq}, \eqref{eq:Beq}, \eqref{eq:dxvBest} and repeated use of H\"older's inequality and the Sobolev embedding $H^2 \hookrightarrow L^\infty$: for example, we estimate the pairing
\begin{align}
	\int_0^T  (\bs B \bs D(\bs v), \bs A)  ds
	&\leq \int_0^T \| \bs B \|_{L^\infty} \| \bs v \|_{H^1} \| A \| ds \\
	&\leq \| \bs B \|_{L^2 H^2} \| \bs v \|_{L^\infty H^1} \| A \|_{L^2 L^2}.
\end{align}
Moreover, we can then conclude that
\begin{align}
\p_t \bs v + P((\bs v \cdot \nabla)\bs v) &= \Delta \bs v + 2a P(\div \bs S(\bs B)) + \bs f, \\
\diam{\bs B} + \bs R(\bs B) &= \Delta \bs B, \label{eq:Leray}
\end{align}
a.e. on $[0,T] \times \bbT^2$ where
$P : L^2_{0} \rar L^2_{0,\div}$
is the Leray projector. The existence and uniqueness of the pressure $p$ then follows immediately from the Helmholtz decomposition of $L_0^2$ vector fields on $\bbT^2$.

We now assume the additional regularity \eqref{eq:H2regularity}. The argument leading to \eqref{eq:wCest} also applies to $\bs w = \p_t \bs v$ and $\bs C = \p_t \bs B$ resulting in the additional estimate
\begin{align}
	\| \p_t \bs v \|_{L^\infty L^2} &+ \| \p_t \bs B \|_{L^\infty L^2} +
\| \p_t \bs v \|_{L^2 H^1} + \| \p_t \bs B \|_{L^2 H^1} \\
&\leq C_{1/3}' \bigl (\| \bs v_0 \|_{H^2}, \| \psi(\bs B_0) \|_{L^1}, \| \bs B_0 \|_{H^2}, \| \bs f \|_{H^1 L^2} \bigr ), \label{eq:dtest}
\end{align}
where we used \eqref{eq:Leray}, \eqref{eq:H2regularity} and Sobolev embedding to bound $\| \p_t \bs v(0) \|$ and $\| \p_t \bs B(0)\|$. By \eqref{eq:Leray}, \eqref{eq:dtest}, Sobolev embedding and interpolation we conclude, for a.e. $t \in [0,T]$,
\begin{align}
\| \Delta \bs v \| &\leq C ( \| \p_t \bs v \| + \| \bs v \|_{L^4} \| \nabla \bs v \|_{L^4}
+ \| \nabla \bs B \| + \| \bs B \|_{L^4} \| \nabla \bs B \|_{L^4} ) \\
&\leq C ( \| \p_t \bs v \| + \| \bs v \|_{H^1}^{3/2} \| \Delta \bs v \|^{1/2}
+ \| \nabla \bs B \| + \| \bs B \|_{H^1}^{3/2} \| \Delta \bs B \|^{1/2}) \\
&\leq C_{2/3}' + \frac{1}{4} \| \Delta \bs v \| + \frac{1}{4} \| \Delta \bs B \|, \label{eq:deltav}
\end{align}
and similarly
\begin{align}
\|  \Delta \bs B \| \leq  C_{2/3}' + \frac{1}{4} \| \Delta \bs v \| + \frac{1}{4} \| \Delta \bs B \|. \label{eq:deltaB}
\end{align}
Summing \eqref{eq:deltav} and \eqref{eq:deltaB} yields, for a.e. $t \in [0,T]$,
\begin{align}
	\| \Delta \bs v \| + \| \Delta \bs B \| \leq 2 C'_{2/3}, \label{eq:H2est}
\end{align}
proving the $L^\infty H^2$ bound. The bound for $\p_t^2 \bs v$ and $\p_t^2 \bs B$ then follows from \eqref{eq:dxveq} with $\bs w = \p_t \bs v$, \eqref{eq:dxBeq} with $\bs C = \p_t \bs B$, \eqref{eq:dtest}, \eqref{eq:H2est} and repeated applications of H\"older's inequality and Sobolev embedding: for example we estimate the pairing
\begin{align}
	\int_0^T (\bs B \bs D(\p_s \bs v), \bs A) ds
	&\leq \| \bs B \|_{L^\infty L^\infty} \| \p_t \bs v \|_{L^2 H^1} \| \bs A \|_{L^2 L^2} \\
	&\leq \| \bs B \|_{L^\infty H^2} \| \p_t \bs v \|_{L^2 H^1} \| \bs A \|_{L^2 H^1}.
\end{align}
This concludes the proof.
\end{proof}

Via an induction argument, the Leibniz formula and repeated use of H\"older's inequality, Sobolev embedding and interpolation (as in the proof of Lemma \ref{l:reg}), we have the follow generalization of Lemma \ref{l:reg}.

\begin{lem}\label{l:higherregu}
	Let $T > 0$, and let $m \in \bbN \cup \{0\}$. Let $\beta \in (0,1), \delta_1, \delta_2 \geq 0$, $a \in \bbR$,
	\begin{align}
		\bs v_0 &\in H^{2m+1}_{0,\div}, \\
		\p_t^j \bs f &\in L^2(0,T; H^{2m-2j}_{0,\div}), \quad j = 0,\ldots,m,
	\end{align}
	and $\bs B_0 : \bbT^2 \rar \mbox{Sym}_+$ such that $\log \det \bs B_0 \in L^1$,  $\bs B_0 \in H^{2m+1}$. Let $(\bs v, \bs B)$ be the unique weak solution to \eqref{eq:divergence}--\eqref{eq:initial} on $[0,T] \times \bbT^2$ such that
	\begin{align}
		\bs v(0) = \bs v_0, \quad \bs B(0) = \bs B_0.
	\end{align}
	Then for $k = 0,\ldots,m+1$,
	\begin{align}
		\begin{split}
			\p_t^k \bs v &\in L^2(0,T; H^{2m+2-2k}_{0,\div}), \\
			\p_t^k \bs B &\in L^2(0,T; H^{2m+2-2k}).
		\end{split}.
	\end{align}
\end{lem}

\subsection{Proof of Theorem \ref{t:main1}}

Let $T > 0$. By \eqref{eq:pressure} and Lemma \ref{l:higherregu} with $m = 0, 1, \ldots,$ we conclude there exist unique smooth $\bs v$, $\bs B$ and $p$ satisfying \eqref{eq:divergence}--\eqref{eq:initial} on $[0,T]$ (the regularity of $p$ following from that of $\bs v$ and $\bs B$ and \eqref{eq:momentum}). Since $T > 0$ was arbitrary we conclude that there exist unique smooth $\bs v$, $\bs B$ and $p$ satisfying  \eqref{eq:divergence}--\eqref{eq:initial} on $[0,\infty)$. Moreover, $\bs B$ is positive definite almost everywhere on $[0,\infty) \times \bbT^2$.

We now prove that $\bs B$ is positive definite everywhere by performing a calculation similar to that done for the standard Oldroyd-B model in \cite{ConstKliegl2dOldroyd}. We write
\begin{align}
	\bs B = \begin{pmatrix}
		\frac{g}{2} + e & f \\
		f & \frac{g}{2} - e
	\end{pmatrix}, \quad
a\bs D(\bs v) + \bs W(\bs v) =
\begin{pmatrix}
	\alpha & \beta \\
	\gamma &-\alpha
\end{pmatrix}.
\end{align}
From \eqref{eq:tensb} we conclude, with $D_t = \p_t + \bs v \cdot \nabla$, that
\begin{align}
	\begin{split}
D_t e &= \Delta e - (\delta_1 + (g - 1)\delta_2) e + (\beta - \gamma)f + \alpha g, \\
D_t f &= \Delta f - (\delta_1 + (g - 1)\delta_2) f - (\beta - \gamma)e + \frac{1}{2}(\beta + \gamma) g, \\
D_t g &= \Delta g - (\delta_1 + (g - 1)\delta_2) g + 4\alpha e + 2(\beta + \gamma) f \\
&\quad + 2\delta_1 + 2\delta_2\Bigl (\frac{g^2}{4} - (e^2 + f^2) \Bigr ).
	\end{split} \label{eq:efgeq}
\end{align}
The maximum and minimum eigenvalues of $\bs B$ are given by
\begin{align}
\lambda_{max} = \frac{g}{2} + (e^2 + f^2)^{1/2}, \quad \lambda_{min} =
\frac{g}{2} - (e^2 + f^2)^{1/2}.
\end{align}
We note that
\begin{align}
\Delta(e^2 + f^2)^{1/2} &- (e^2+f^2)^{-1/2}e\Delta e - (e^2 + f^2)^{-1/2} f \Delta f \\
&= (e^2 + f^2)^{-1/2} \sum_{j = 1}^2 (e \p_j f - f \p_j e)^2 \geq 0. \label{eq:deltaineq}
\end{align}
Using \eqref{eq:efgeq} and \eqref{eq:deltaineq} we conclude
\begin{align}
D_t &\lambda_{min} + \Delta \lambda_{min} \\
&= -[ \delta_1 + (g - 1) \delta_2 + 2 \alpha e (e^2 + f^2)^{-1/2} +
(\beta + \gamma) f (e^2 + f^2)^{-1/2}
] \lambda_{min} \\
&\quad + \delta_1 + \delta_2 \lambda_{max}\lambda_{min} \\
&\quad + \Delta(e^2 + f^2)^{1/2} - (e^2+f^2)^{-1/2}e\Delta e - (e^2 + f^2)^{-1/2} f \Delta f \\
&\geq  -[ \delta_1 + (g - 1) \delta_2 + 2 \alpha e (e^2 + f^2)^{-1/2} +
(\beta + \gamma) f (e^2 + f^2)^{-1/2}
] \lambda_{min} \\
&\quad + \delta_1 + \delta_2 \lambda_{max}\lambda_{min}.
\end{align}
Since $\lambda_{min}(0,x) >  0$ by assumption that $\bs B_0$ is positive definite, the maximum principle implies that $\lambda_{min}(t,x) > 0$ on $[0,\infty)$. This proves that $\bs B$ is positive definite on $[0,\infty)$ and concludes the proof of Theorem \ref{t:main1}.
\qed

\appendix

\section{Existence of weak solutions}

In this appendix we prove the existence of weak solutions via an approximation scheme.

\subsection{Galerkin system}

An orthonormal basis for $L^2_{0,\div}$ is given by $$\{ n^{\perp} e^{2\pi i n\cdot x} / |n| \}_{n \in \bbZ^2 \backslash (0,0)}$$ where $n^\perp = -n_2  \bs i + n_1 \bs j$. We denote the projection of $\bs v \in L^2_{0,\div}$ onto
\begin{align*}
	\mbox{Span}\{ n^{\perp} e^{2\pi i n\cdot x} / |n| \}_{|n| \leq k}
\end{align*}
by $P_k \bs v$. We denote the projection of  $\bs B \in L^2$ onto
\begin{align}
	\mbox{Span}
		\Bigl \{
		\begin{pmatrix}
			1 & 0 \\
			0 & 0
		\end{pmatrix}
		e^{2\pi in \cdot x}, \frac{1}{\sqrt2}  \begin{pmatrix}
			1 & 0 \\
			0 & 1
		\end{pmatrix}
		e^{2\pi in \cdot x},
		\begin{pmatrix}
			0 & 0 \\
			0 & 1
		\end{pmatrix}
		e^{2\pi in \cdot x}
		\Bigr \}_{|n| \leq k}
	\end{align}
by $Q_k \bs B$.

For $\eps > 0$, $\bs A \in \bbR^{2 \times 2}$, $\bs A^T = \bs A$, we denote the minimal eigenvalue of $\bs A$ by $\Lambda(\bs A)$ and define
\begin{align}
	\rho_\eps(\bs A) = \frac{\max\{0, \Lambda(\bs A) - \eps \}}{\Lambda(\bs A)(1+ \eps|\bs A|^3)}.
\end{align}

Let $\eps > 0$, $k \in \bbN$. By an application of Carath\'{e}odory's theorem for ordinary differential equations, it follows that there exist $T^* \in (0,T]$ and unique
\begin{align}
	\bs v^k_{\eps} &\in 	\mbox{Span}\{ n^{\perp} e^{2\pi i n\cdot x} / |n| \}_{|n| \leq k},
	\\
	\bs B^k_{\eps} &\in
	\mbox{Span}
	\Bigl \{
	\begin{pmatrix}
		1 & 0 \\
		0 & 0
	\end{pmatrix}
	e^{2\pi in \cdot x}, \frac{1}{\sqrt2}  \begin{pmatrix}
		1 & 0 \\
		0 & 1
	\end{pmatrix}
	e^{2\pi in \cdot x},
	\begin{pmatrix}
		0 & 0 \\
		0 & 1
	\end{pmatrix}
	e^{2\pi in \cdot x}
	\Bigr \}_{|n| \leq k},
\end{align}
with absolutely continuous Fourier coefficients solving the following Galerkin system on $[0,T^*]$:
\begin{align}
\p_t \bs v^k_\eps &+ P_k(\bs v^k_\eps \cdot \nabla \bs v_\eps^k) - \Delta \bs v_\eps^k \\
&= 2a P_k [\div (\rho_\eps(\bs B^k_\eps) \bs S(\bs B^k_\eps) )
] + P_k \bs f, \label{eq:Galv}\\
\p_t \bs B^k_\eps &+ Q_k(\bs v^k_\eps \cdot \nabla \bs B_\eps^k) - \Delta \bs B_\eps^k \\
&=
a Q_k[\rho_\eps(\bs B^k_\eps) \bs D(\bs v^k_\eps) \bs B^k_\eps
+ \rho_\eps(\bs B^k_\eps) \bs B^k_\eps \bs D(\bs v^k_\eps) ] \\
&+  Q_k[\rho_\eps(\bs B^k_\eps) \bs W(\bs v^k_\eps) \bs B^k_\eps
- \rho_\eps(\bs B^k_\eps) \bs B^k_\eps \bs W(\bs v^k_\eps) ] \\
& - Q_k[\rho_\eps(\bs B^k_\eps)\bs R(\bs B^k_\eps) ] \label{eq:GalB},
\end{align}
with initial conditions
\begin{gather}
	\bs v^k_\eps(0,x) = P_k \bs v_0(x), \quad \bs B_\eps^k(0,x) =  Q_k \bs B_{0,\eps}(x), \\
	\bs B_{0,\eps}(x) =
	\begin{cases}
		\bs B_0(x) &\mbox{ if } \Lambda(\bs B_0(x)) > \eps, \\
		\bs I &\mbox{ otherwise}.
	\end{cases}
\end{gather}

\begin{lem}\label{l:Genergy}
	For all $\eps > 0$, $k \in \bbN$, there exist constants $C_0(\eps)$ and $C_1$ (absolute) such that the solution $(\bs v^k_\eps, \bs B^k_\eps)$ to the Galerkin system \eqref{eq:Galv}, \eqref{eq:GalB} satisfies
	\begin{align}
	\sup_{t \in [0,T^*]} ( \| \bs v^k_\eps(t)\|^2 &+ \| \bs B^k_{\eps}(t) \|^2 )
	+ \int_0^{T^*} (\| \nabla \bs v^k_\eps \|^2 + \| \nabla \bs B^k_\eps \|^2 ) dt \\
	&\leq C_0(\eps) + \| \bs v_0 \|^2 + \| \bs B_0 \|^2 + C_1 \int_0^{T^*} \| \bs f \|_{H^{-1}_{0,\div}}^2 dt. \label{eq:vBGen}
	\end{align}
	Moreover, denoting the right-hand side of \eqref{eq:vBGen} by $D$, there exists a constant $C_3(D,\eps)$ such that
	\begin{align}
		\int_0^{T^*} (\| \p_t \bs v^k_{\eps} \|^2_{H^{-1}_{0,\div}} +
		\| \p_t \bs B^k_{\eps} \|_{H^{-1}}^2 ) dt \leq C_3(D,\eps). \label{eq:dtGen}
	\end{align}
\end{lem}

\begin{proof}
Pairing \eqref{eq:Galv} with $\bs v^k_\eps$ and \eqref{eq:GalB} with $\bs B^k_\eps$ and summing we obtain the identity
\begin{align}
	\frac{1}{2} \frac{d}{dt} ( \| \bs v^k_\eps \|^2 &+ \| \bs B^k_\eps \|^2 )
	+ \| \nabla \bs v^k_\eps \|^2 + \| \nabla \bs B^k_\eps \|^2 \\
	&= -2a ( \rho_\eps (\bs B^k_\eps) \bs S(\bs B^k_\eps), \bs D(\bs v^k_\eps)) +
	\ang{f, \bs v^k_\eps} \\
	&+ 2a [ (\rho_\eps (\bs B^k_\eps) \bs B^k_\eps \bs D(\bs v^k_\eps), \bs B^k_\eps)
	- ( \rho_\eps (\bs v^k_\eps) \bs R(\bs B^k_\eps), \bs B^k_\eps) ]. \label{eq:GalEn}
\end{align}
The definitions of $\rho_\eps$, $\bs S$, and $\bs R$ imply the estimate
\begin{align}
\rho_\eps(\bs B^k_\eps)[ |\bs S(\bs B^k_\eps)| + |\bs B^k_\eps|^2 +
|\bs B^k_\eps||\bs R(\bs B^k_\eps)| ] \leq C \frac{1 + |\bs B^k_\eps|^3}{1 +
\eps |\bs B^k_\eps|^3} \leq C_0(\eps). \label{eq:Nonlinest}
\end{align}
Integrating \eqref{eq:GalEn} from $0$ to $t\in [0,T^*]$ and
using \eqref{eq:Nonlinest} to bound the right-hand side of \eqref{eq:GalEn} we obtain the estimate \eqref{eq:vBGen}.

We now estimate $\p_t \bs v^k_\eps$ and $\p_t \bs B^k_\eps$.
Let $\bs w \in H^1_{0,\div}$. By H\"older's inequality, Sobolev embedding $H^{1/2} \hookrightarrow L^4$ and interpolation we obtain the estimates:
\begin{align}
|(P_k ((\bs v^k_\eps \cdot \nabla) \bs v^k_\eps), \bs w)| &=
|(\bs v^k_\eps, (\bs v^k_\eps \cdot \nabla) P_k \bs w)| \\
&\leq \| \bs v^k_\eps \|_{L^4}^2 \| \nabla \bs w \| \\
&\leq C \| \bs v^k_\eps \| \| \grad \bs v^k_\eps \| \| \grad \bs w \|, \\
|(P_k[\div(\rho_\eps(\bs B^k_\eps) \bs S(\bs B^k_\eps)  )], \bs w)| &
=  |(\rho_\eps(\bs B^k_\eps) \bs S(\bs B^k_\eps), \nabla P_k \bs w )| \\
&\leq C \int_{\bbT^2} (1 + |\bs B^k_\eps| + |\bs B^k_\eps|^2) |\nabla P_k \bs w| dx \\
&\leq C (1 + \| \bs B^k_\eps \| + \| \bs B^k_\eps \| \| \nabla \bs B^k_\eps \| ) \| \nabla w \|.
\end{align}
These estimates and \eqref{eq:Galv} imply
\begin{align*}
	\int_0^{T^*} \| \p_t \bs v \|^2_{H^{-1}_{0,\div}} dt \leq C_{7/3}(D).
\end{align*}
Similarly, if $\bs A \in H^{-1}$, we obtain
\begin{align}
	|(Q_k(
	(\bs v_\eps^k \cdot \nabla) \bs B^k_\eps), \bs A
	)| &\leq C \| \bs v^k_\eps \|^{1/2} \| \bs v^k_\eps \|^{1/2}
	\| \bs B^k_\eps \|^{1/2} \| \bs B^k_\eps \|^{1/2} \| \nabla \bs A \|, \\
	| (Q_k(
	\rho_\eps(\bs B^k_\eps) \bs R(\bs B^k_\eps)
	), \bs A
	)
	| &\leq C (1 + \| \bs B^k_\eps \| + \| \bs B^k_\eps \| \| \nabla \bs B^k_\eps \| ) \| \nabla \bs A \|,
\end{align}
and since $|\rho_\eps(\bs B^k_\eps)|\bs B_{\eps}^k| \leq C_{8/3}(\eps)$, we obtain
\begin{align}
|(a Q_k[
\rho_\eps(\bs B^k_\eps ) \bs D(\bs v^k_\eps ) \bs B^k_\eps &+
\rho_\eps(\bs B^k_\eps ) \bs B^k_\eps \bs D(\bs v^k_\eps )
],
\bs A) | \\
&\leq 2|a| C_{8/3}(\eps) \| \nabla \bs v^k_\eps \| \| \bs A \|, \\
|(a Q_k[
\rho_\eps(\bs B^k_\eps ) \bs W(\bs v^k_\eps ) \bs B^k_\eps &-
\rho_\eps(\bs B^k_\eps ) \bs B^k_\eps \bs W(\bs v^k_\eps )
],
\bs A) | \\
&\leq 2|a| C_{8/3}(\eps)  \| \nabla \bs v^k_\eps \| \| \bs A \|.
\end{align}
These estimates and \eqref{eq:GalB} imply that
\begin{align}
	\int_0^{T^*} (\| \p_t \bs v \|^2_{H^{-1}_{0,\div}} + \| \p_t \bs B^k_\eps \|^2_{H^{-1}}) dt \leq C_3(D,\eps),
\end{align}
completing the proof.
\end{proof}

\begin{lem}\label{l:klimit}
For all $\eps > 0$, there exist
\begin{align}
\bs v_\eps &\in L^\infty(0,T; L^2_{0,\div}) \cap L^2(0,T; H^1_{0,\div}) \cap C([0,T];
H^{-1}_{0,\div}), \\
\bs B_\eps &\in L^\infty(0,T; L^2) \cap L^2(0,T; H^1) \cap C([0,T];
H^{-1}),
\end{align}
such that $\bs v_{\eps}(0,x) = \bs v_0(x)$, $\bs B_\eps(0,x) = \bs B_{0,\eps}(x)$,
\begin{align}
	\p_t \bs v_{\eps} &\in L^2(0,T; H^{-1}_{0,\div}), \\
 	\p_t \bs B_{\eps} &\in L^2(0,T; H^{-1}),
\end{align}
for all $\bs \varphi \in L^2(0,T; H^1_{0,\div})$,
\begin{align}
	\int_0^T &[\ang{\p_t \bs v_{\eps}, \bs \varphi} + ((\bs v_\eps \cdot \nabla) \bs v_\eps, \bs \varphi) ] dt \\
	&= -\int_0^T (\nabla \bs v_{\eps}, \nabla \bs \varphi) - 2a ((\rho_{\eps}(\bs B_{\eps})) \bs S(\bs B_\eps), \bs D( \bs \varphi)) + \ang{\bs f, \bs \varphi} dt, \label{eq:epsv}
\end{align}
and for all $\bs A \in L^2(0,T;H^{1})$, we have
\begin{align}
	\int_0^T &[\ang{\p_t \bs B_{\eps}, \bs A } + ((\bs v_\eps \cdot \nabla ) \bs B_{\eps} + 2 \rho_{\eps}(\bs B_{\eps})\bs B\bs W(\bs v) - 2 a \rho(_{\eps}\bs B_{\eps})\bs B \bs D(\bs v), \bs A) ] dt \\
	&+ \int_0^T (\rho_{\eps}(\bs B_{\eps})\bs R(\bs B_{\eps}), \bs A) dt =
	- \int_0^T (\nabla \bs B, \nabla \bs A) dt. \label{eq:epsB}
\end{align}

Moreover, for almost every $(t,x) \in [0,T] \times \bbT^2$, the minimal eigenvalue of $\bs B_\eps$ satisfies
\begin{align}
	\Lambda(\bs B_{\eps}) \geq \eps. \label{eq:Bepslower}
\end{align}
\end{lem}

\begin{proof}
For $\eps > 0$, $k \in \bbN$, let $(\bs v^k_{\eps}, \bs B^k_{\eps})$ be the solution to the Galerkin system \eqref{eq:Galv}, \eqref{eq:GalB}. By Lemma \ref{l:Genergy} and the Banach-Alaoglu theorem there exist subsequences (which we will not relabel) and $(\bs v_\eps, \bs B_{\eps})$ such that:
\begin{align}
\bs v^k_\eps &\overset{\ast}{\rightharpoonup} \bs v_\eps \quad \mbox{weakly in } L^\infty(0,T;L^2), \\
\bs v^k_\eps &\rightharpoonup \bs v_\eps \quad \mbox{weakly in } L^2(0,T;H^1_{0,\div}), \\
\p_t \bs v^k_\eps &\overset{\ast}{\rightharpoonup} \p_t \bs v_\eps \quad \mbox{weakly in } L^\infty(0,T;H^{-1}_{0,\div}), \\
\bs B^k_\eps &\overset{\ast}{\rightharpoonup} \bs B_\eps \quad \mbox{weakly in } L^\infty(0,T;L^2), \\
\bs B^k_\eps &\rightharpoonup \bs B_\eps \quad \mbox{weakly in } L^2(0,T;H^1), \\
\p_t \bs B^k_\eps &\overset{\ast}{\rightharpoonup} \p_t \bs B_\eps \quad \mbox{weakly in } L^\infty(0,T;H^{-1}).
\end{align}
By the Aubin--Lions lemma and extracting further subsequences if necessary, we have also
\begin{align}
\bs v^k_{\eps} &\rar \bs v_\eps \quad \mbox{strongly in }L^2(0,T;H^{1/2}_{0,\div}) \cap C([0,T]; H^{-1}_{0,\div}), \\
\bs B^k_{\eps} &\rar \bs B_\eps \quad \mbox{strongly in }L^2(0,T;H^{1/2}) \cap C([0,T]; H^{-1}), \\
\rho_{\eps}(\bs B^k_{\eps}) &\rar \rho_{\eps}(\bs B_{\eps}) \quad \mbox{a.e. on } [0,T] \times \bbT^2.
\end{align}
Using \eqref{eq:Galv}, \eqref{eq:GalB}, the previous convergence properties and standard arguments we can conclude that $(\bs v_\eps,\bs B_{\eps})$ satisfies \eqref{eq:epsv},  \eqref{eq:epsB} and
$\bs v_{\eps}(0,x) = \bs v_0(x)$, $\bs B_{\eps}(0,x) = \bs B_{\eps}(x)$.

The proof of \eqref{eq:Bepslower} is the same as the $3d$ case in \cite{BBJ21}, but we will briefly sketch it for completeness. Let $\bs z \in \bbQ^2$. In \eqref{eq:epsB} we choose
\begin{align}
	\bs A(t,x) = \chi_{[0,\tau]}(t)(\bs B_{\eps}(t,x) \bs z \cdot \bs z - \eps |\bs z|^2)_-(\bs z \otimes \bs z) \in L^2(0,T;H^1).
\end{align}
This choice of $\bs A$, and the facts that $\div \bs v_{\eps} = 0$ and $\rho_{\eps}(\bs B_{\eps})\bs A = \bs 0$ imply that the only nonzero terms appearing in \eqref{eq:epsB} are
\begin{align}
\int_0^T \ang{\p_t \bs B_{\eps}, \bs A} dt &=
\int_0^\tau \frac{d}{dt} \frac{1}{2} \| (\bs B_{\eps} \bs z \cdot \bs z - \eps |\bs z|^2)_- \|^2 dt \\
&= \frac{1}{2}  \| (\bs B_{\eps}(\tau) \bs z \cdot \bs z - \eps |\bs z|^2)_- \|^2
- \frac{1}{2}  \| (\bs B_{0,\eps} \bs z \cdot \bs z - \eps |\bs z|^2)_- \|^2 \\
&= \frac{1}{2}  \| (\bs B_{\eps}(\tau) \bs z \cdot \bs z - \eps |\bs z|^2)_- \|^2,
\end{align}
and
\begin{align}
\int_0^T (\nabla \bs B_{\eps}, \nabla \bs A) dt =
\int_0^\tau \| \nabla ( \bs B_{\eps} \bs z \cdot \bs z - \eps |\bs z|^2)_-\|^2 dt.
\end{align}
By \eqref{eq:epsB}
\begin{align}
\| (\bs B_{\eps}(\tau) \bs z \cdot \bs z - \eps |\bs z|^2)_- \|^2 = 0, \quad \tau \in [0,T].
\end{align}
Thus, $\bs B(t,x)\bs z \cdot \bs z \geq \eps |\bs z|^2$ almost everywhere on $[0,T] \times \bbT^2$.
 We conclude that for a.e. $(t,x) \in [0,T] \times \bbT^2$, for all $\bs z \in \bbR^2$, we have $\bs B(t,x)\bs z \cdot \bs z \geq \eps |\bs z|^2$, proving \eqref{eq:Bepslower}.
\end{proof}

\subsection{An energy identity}

We now derive an energy identity for the $\eps$-modified equations reminiscent of the energy balance identity \eqref{eq:energybalance} for the equations of motion.

\begin{lem}\label{l:epsest}
Let $\eps > 0$ and $(\bs v_{\eps}, \bs B_{\eps})$ be as in Lemma \ref{l:klimit}. Then for all $\tau \in [0,T]$,
\begin{align}
\frac{1}{2} & \| \bs v_{\eps}(\tau) \|^2 + \int_{\bbT^2} \psi(\bs B_{\eps}(\tau)) dx +
\int_0^\tau \| \nabla \bs v_{\eps} \|^2 dt  \\
&+  \int_0^\tau [ (1-\beta) \| \bs B_{\eps}^{-1/2} \nabla \bs B_{\eps} \bs B_{\eps}^{-1/2} \|^2
+ \beta \| \nabla \bs B_{\eps} \|^2 ] dt \\
&+ \int_0^\tau (1-\beta) \delta_1 \|\rho_{\eps}(\bs B_{\eps})^{1/2} (\bs B_{\eps}^{1/2} - \bs B_{\eps}^{-1/2}) \|^2 dt \\
&+
\int_0^\tau \beta \delta_2 \|\rho_{\eps}(\bs B_{\eps})^{1/2} (\bs B_{\eps}^{3/2} - \bs B_{\eps}^{1/2}) \|^2
 dt \\
&+ \int_0^\tau (\beta \delta_1 + (1-\beta) \delta_2) \| \rho_{\eps}(\bs B_{\eps})^{1/2}(\bs B_{\eps} - \bs I) \|^2 dt
\\& = \frac{1}{2} \| \bs v_0 \|^2 + \int_{\bbT^2} \psi(\bs B_{0,\eps}) dx
+ \int_0^\tau \ang{\bs f, \bs v} dt \label{eq:epsid}.
\end{align}
\end{lem}

\begin{proof}
By \eqref{eq:Bepslower} it follows that for a.e. $(t,x) \in [0,T] \times \bbT^2$, $\bs B^{-1}_\eps$ exists and
\begin{align*}
	|\bs B^{-1}_{\eps}(t,x)| \leq \frac{1}{\eps}.
\end{align*}
Then $\nabla \bs B^{-1}_\eps = -\bs B_{\eps}^{-1} \nabla \bs B_{\eps} \bs B_{\eps}^{-1}$ implies
\begin{align}
	\int_0^T \| \nabla \bs B^{-1}_{\eps} \|^2 dt \leq \frac{1}{\eps^4}
	\int_0^T \| \nabla \bs B_{\eps} \|^2 dt. \label{eq:Bepsinvest}
\end{align}

Let
\begin{align*}
\bs J_{\eps} = (1-\beta)(\bs I - \bs B_{\eps}^{-1}) + \beta (\bs B_{\eps}  - I).
\end{align*}
By \eqref{eq:Bepsinvest}, $\bs J_{\eps} \in L^2(0,T; H^1)$. Since $\frac{\p \psi(\bs B_{\eps})}{\p \bs B_{\eps}} = \bs J_{\eps}$ and $\div \bs v_{\eps} = 0$, one readily verifies the identities
\begin{align}
\ang{\p_t \bs B_{\eps}, \bs J_{\eps}} &= \frac{d}{dt} \int_{\bbT^2} \psi(\bs B_{\eps}) dx, \\
((\bs v_{\eps} \cdot \nabla) \bs B_{\eps}, \bs J_{\eps}) &=
\int_{\bbT^2} \bs v_{\eps} \cdot \nabla \psi(\bs B_{\eps}) dx = 0, \\
(\nabla \bs B_{\eps}, \nabla \bs J_{\eps}) &=
(1-\beta) \| \bs B_{\eps}^{-1/2} \nabla \bs B_{\eps} \bs B_{\eps}^{-1/2} \|^2
+ \beta \| \nabla \bs B_{\eps} \|^2, \\
(\rho_\eps(\bs B_{\eps}) \bs R(\bs B_{\eps}), \bs J_{\eps} ) &=
 (1-\beta) \delta_1 \|\rho_{\eps}(\bs B_{\eps})^{1/2} (\bs B_{\eps}^{1/2} - \bs B_{\eps}^{-1/2}) \|^2 \\
&+
\beta \delta_2 \|\rho_{\eps}(\bs B_{\eps})^{1/2} (\bs B_{\eps}^{3/2} - \bs B_{\eps}^{1/2}) \|^2
 \\
&+ (\beta \delta_1 + (1-\beta) \delta_2) \| \rho_{\eps}(\bs B_{\eps})^{1/2}(\bs B_{\eps} - \bs I) \|^2.
\end{align}
Since $$\bs B_{\eps} \bs J_{\eps} = \bs S(\bs B_{\eps}) = \bs J_{\eps} \bs B_{\eps}$$ we also have
\begin{align}
2(\rho_{\eps}(\bs B_{\eps}) \bs B_{\eps} \bs W(\bs v_{\eps}), \bs J_{\eps}) &= 0, \\
2a (\rho_{\eps}(\bs B_{\eps}) \bs B_{\eps} \bs D(\bs v_{\eps}), \bs J_{\eps}) &=
2a (\rho_{\eps}(\bs B_{\eps}) \bs  D(\bs v_{\eps}), \bs S(\bs B_{\eps}).
\end{align}
Using the previous identities and \eqref{eq:epsB} with $\bs A = \chi_{[0,\tau]}\bs J_{\eps}$ we obtain the identity
\begin{align}
	 \int_{\bbT^2} \psi&(\bs B_{\eps}(\tau)) dx
	+  \int_0^\tau [ (1-\beta) \| \bs B_{\eps}^{-1/2} \nabla \bs B_{\eps} \bs B_{\eps}^{-1/2} \|^2
	+ \beta \| \nabla \bs B_{\eps} \|^2 ] dt \\
	&+ \int_0^\tau  (1-\beta) \delta_1 \|\rho_{\eps}(\bs B_{\eps})^{1/2} (\bs B_{\eps}^{1/2} - \bs B_{\eps}^{-1/2}) \|^2 dt \\
	&+
	\int_0^\tau \beta \delta_2 \|\rho_{\eps}(\bs B_{\eps})^{1/2} (\bs B_{\eps}^{3/2} - \bs B_{\eps}^{1/2}) \|^2
	 dt \\
	&+ \int_0^\tau (\beta \delta_1 + (1-\beta) \delta_2) \| \rho_{\eps}(\bs B_{\eps})^{1/2}(\bs B_{\eps} - \bs I) \|^2 dt
	\\& =  \int_{\bbT^2} \psi(\bs B_{0,\eps}) dx +
	\int_0^\tau 2a (\rho_{\eps}(\bs B_{\eps}) \bs S(\bs B_{\eps}), \bs D(\bs v_{\eps}) dt. \label{eq:Bepsid}
\end{align}
Using \eqref{eq:epsv} with $\bs \varphi = \chi_{[0,\tau]} \bs v_{\eps}$ we obtain
\begin{align}
\frac{1}{2} \| \bs v_{\eps} \|^2 &+ \int_0^\tau \| \nabla \bs v_{\eps} \|^2 dt \\
&= \frac{1}{2} \| \bs v_0 \|^2 - \int_0^\tau 2a (\rho_{\eps}(\bs B_{\eps}) \bs S(\bs B_{\eps}), \bs D(\bs v_{\eps})). \label{eq:vepsid}
\end{align}
Adding \eqref{eq:Bepsid} to \eqref{eq:vepsid} yields \eqref{eq:epsid}.
\end{proof}

\subsection{Conclusion of the proof of Theorem \ref{t:main2}}

We now conclude the proof of Theorem \ref{t:main2} by establishing the existence of a weak solution. For $\eps >  0$, let $(\bs v_\eps, \bs B_\eps)$ be as in Lemma \ref{l:klimit}. By the definition of $\bs B_{0,\eps}$, the fact $\psi(\bs I) = 0$ and Lemma \ref{l:epsest} we conclude for all $\tau \in [0,T]$
\begin{align}
	\frac{1}{2} \sup_{\tau \in [0,T]} \| \bs v_\eps(\tau) \|^2& +\sup_{\tau \in [0,T]}
	\int_{\bbT^2} \psi(\bs B_\eps(\tau)) dx + \frac{1}{2} \int_0^T ( \| \nabla \bs v_\eps \|^2 +
	\| \nabla \bs B_\eps \|^2 ) dt \\
	& \leq \frac{1}{2} \| \bs v_0 \|^2 + \int_{\bbT^2} \psi(\bs B_0) dx + C \int_0^T \| \bs f \|^2_{H^{-1}_{0,\div}} dt \label{eq:epsest}
\end{align}
where $C$ is an absolute constant.
Note that since $\int_{\bbT^2} |\log \det \bs B_0| dx < \infty$, the right-hand side of \eqref{eq:epsest} is finite and independent of $\eps$. Using \eqref{eq:epsv}, \eqref{eq:epsB}, \eqref{eq:epsest} and arguing as in the proof of \eqref{eq:dtGen} we obtain
\begin{align}
\| \p_t \bs v_{\eps} \|_{L^2 H^{-1}_{0,\div}} &\leq B(\bs v_0, \bs B_0, \bs f), \label{eq:dtveps} \\
\| \p_t \bs B_{\eps} \|_{[L^2 H^1 \cap L^4L^4]^*} &\leq B(\bs v_0, \bs B_0, \bs f),  \label{eq:dtBeps} \\
\forall \delta > 0, \quad \| \p_t \bs B_{\eps} \|_{L^2 H^{-1-\delta}} &\leq B(\bs v_0, \bs B_0, \bs f, \delta). \label{eq:dtBepsdelta}
\end{align}
For example, to establish \eqref{eq:dtBeps} we estimate the pairing
\begin{align}
	\int_0^T (\bs D(\bs v_\eps ) \bs B_\eps +&
	 \bs B_\eps \bs D(\bs v_\eps )
	,
	\bs A) dt \\ &\leq C
	\| \bs v_{\eps} \|_{L^2 H^1} \| \bs B_\eps \|_{L^4 L^4} \| \| \bs A \|_{L^4 L^4} \\
	&\leq C \| \bs v_{\eps} \|_{L^2 H^1} \| \bs B_\eps \|_{L^\infty L^2}^{1/2} \| \bs B_\eps \|_{L^2 H^1}^{1/2}\| \| \bs A \|_{L^4 L^4},
\end{align}
and to establish \eqref{eq:dtBepsdelta} we use Sobolev embedding $H^{1+\delta} \hookrightarrow L^\infty$ to estimate the pairing
\begin{align}
	\int_0^T (\bs D(\bs v_\eps ) \bs B_\eps +&
	\bs B_\eps \bs D(\bs v_\eps )
	,
	\bs A) dt \\ &\leq C(\delta)
	\| \bs v_{\eps} \|_{L^2 H^1} \| \bs B_\eps \|_{L^\infty L^2} \| \| \bs A \|_{L^2 H^{1+\delta}}.
\end{align}

By \eqref{eq:epsest}, \eqref{eq:dtveps}, \eqref{eq:dtBeps}, \eqref{eq:dtBepsdelta} and the Banach-Alaoglu theorem there exist $\eps_j \rar 0$ and $(\bs v, \bs B)$ such that:
\begin{align}
	\bs v_{\eps_j} &\overset{\ast}{\rightharpoonup} \bs v \quad \mbox{weakly in } L^\infty(0,T;L^2), \\
	\bs v_{\eps_j} &\rightharpoonup \bs v \quad \mbox{weakly in } L^2(0,T;H^1_{0,\div}), \\
	\p_t \bs v_{\eps_j} &\overset{\ast}{\rightharpoonup} \p_t \bs v \quad \mbox{weakly in } L^\infty(0,T;H^{-1}_{0,\div}), \\
	\bs B_{\eps_j} &\overset{\ast}{\rightharpoonup} \bs B \quad \mbox{weakly in } L^\infty(0,T;L^2), \\
	\bs B_{\eps_j} &\rightharpoonup \bs B \quad \mbox{weakly in } L^2(0,T;H^1), \\
	\p_t \bs B_{\eps_j} &\overset{\ast}{\rightharpoonup} \p_t \bs B \quad \mbox{weakly in } [L^2(0,T;H^1) \cap L^4(0,T; L^4)]^*, \\
	\forall \delta >  0, \, \p_t \bs B_{\eps_j} &\overset{\ast}{\rightharpoonup} \p_t \bs B \quad \mbox{weakly in } L^\infty(0,T;H^{-1-\delta}).
\end{align}
We note that \eqref{eq:epsest} and these weak convergence properties imply that
\begin{align}
	\| \bs v \|_{L^\infty L^2} &+ \| \bs B \|_{L^\infty L^2} +
	\| \nabla \bs v \|_{L^2 L^2} + \| \nabla \bs B \|_{L^2 L^2} \\
	&\leq C_0 \bigl (\| \bs v_0 \|, \| \psi(\bs B_0) \|_{L^1}, \| \bs f \|_{L^2 H^{-1}_{0,\div}} \bigr ) \label{eq:vBenergyest}.
\end{align}
By the Aubin--Lions lemma and extracting further subsequences if necessary, we have also
\begin{align}
	\bs v_{\eps_j} &\rar \bs v \quad \mbox{strongly in }L^2(0,T;H^{1/2}_{0,\div}) \cap C([0,T]; H^{-1}_{0,\div}), \label{eq:vstrong}\\
	\bs B_{\eps_j} &\rar \bs B \quad \mbox{strongly in }L^2(0,T;H^{1/2}) \cap C([0,T]; H^{-1-\delta}), \, \forall \delta > 0, \label{eq:Bstrong} \\
	\bs B_{\eps_j} &\rar \bs B \quad \mbox{a.e. on } [0,T] \times \bbT^2.
\end{align}

By \eqref{eq:Bepslower} and the continuity of $\Lambda(\cdot)$, $\Lambda(\bs B) \geq 0$ a.e. on $[0,T]\times \bbT^2$. In particular, $\det \bs B_0 \geq 0$ a.e., and by continuity of $\psi(\cdot)$,
\begin{align}
	\psi(\bs B_{\eps_j}) \rar \psi(\bs B) \in [0,\infty]
\end{align}
a.e. on $[0,T] \times \bbT^2$. Fatou's lemma and \eqref{eq:epsest} then imply that
\begin{align}
	\| \psi(\bs B) \|_{L^\infty L^1} < \infty. \label{eq:psibd}
\end{align}
Since $\Lambda(\bs B) = 0$ on a subset of $[0,T] \times \bbT^2$ with positive measure would imply that $\psi(\bs B) = \infty$ on a subset with positive measure, we conclude from \eqref{eq:psibd} that
\begin{align}
	\Lambda(\bs B) > 0 \mbox{ a.e. on } [0,T] \times \bbT^2.
\end{align}
Thus, by the continuity of $\Lambda(\cdot)$,
\begin{align}
	\rho_{\eps_j}(\bs B_{\eps_j}) \rar 1 \mbox{ a.e. on } [0,T] \times \bbT^2.
\end{align}
Using \eqref{eq:epsv}, \eqref{eq:epsB}, the previous convergence results and standard arguments we conclude that $(\bs v, \bs B)$ verify \eqref{eq:veq} and \eqref{eq:Beq}. Finally, the additional regularity $\bs v \in C([0,T]; L^2_{0,\div})$ and $\bs B \in C([0,T]; L^2)$ follow immediately from the facts that
\begin{align}
	\bs v \in L^2(0,T; H^1_{0,\div}), \p_t \bs v \in L^2(0,T; H^{-1}_{0,\div})
\end{align}
and
\begin{align}
	\bs B &\in L^2(0,T; H^1)\cap L^\infty(0,T; L^2), \\
	\p_t \bs B
	&\in [L^2(0,T; H^1) \cap L^4(0,T;L^4)]^*.
\end{align}

\qed

\bigskip

\noindent{\bf Acknowledgement.} Miroslav Bul\'\i{}\v cek and Josef M\'{a}lek acknowledge the support of the project No. 20-11027X financed by Czech science foundation (GACR). Miroslav Bul\'\i{}\v cek and Josef M\'{a}lek are members of the Ne\v{c}as center for mathematical modelling.


\begin{thebibliography}{10}

\bibitem{BBJ21}
M.~Bathory, M.~Bul\'\i{}\v{c}ek, and J.~M\'{a}lek.
\newblock Large data existence theory for three-dimensional unsteady flows of
  rate-type viscoelastic fluids with stress diffusion.
\newblock {\em Adv. Nonlinear Anal.}, 10(1):501--521, 2021.

\bibitem{constantin.p.foias.c:navier-stokes}
P.~Constantin and C.~Foias.
\newblock {\em Navier-Stokes equations}.
\newblock University of Chicago Press, 1988.

\bibitem{ConstKliegl2dOldroyd}
P.~Constantin and M.~Kliegl.
\newblock Note on global regularity for two-dimensional {O}ldroyd-{B} fluids
  with diffusive stress.
\newblock {\em Arch. Ration. Mech. Anal.}, 206(3):725--740, 2012.

\bibitem{divoux.t.fardin.ma.ea:shear}
T.~Divoux, M.~A. Fardin, S.~Manneville, and S.~Lerouge.
\newblock Shear banding of complex fluids.
\newblock {\em Annu. Rev. Fluid Mech.}, 48(1):81--103, 2016.

\bibitem{el-kareh_leal}
A.~W. El-Kareh and L.~G. Leal.
\newblock Existence of solutions for all {D}eborah numbers for a
  non-{N}ewtonian model modified to include diffusion.
\newblock {\em J. Non-Newton. Fluid Mech.}, 33(3):257--287, 1989.

\bibitem{EvansPDE}
L.~C. Evans.
\newblock {\em Partial differential equations}, volume~19 of {\em Graduate
  Studies in Mathematics}.
\newblock American Mathematical Society, Providence, RI, second edition, 2010.

\bibitem{fardin.m.radulescu.o.ea:stress}
M.-A. Fardin, O.~Radulescu, A.~Morozov, O.~Cardoso, J.~Browaeys, and
  S.~Lerouge.
\newblock Stress diffusion in shear banding wormlike micelles.
\newblock {\em J. Rheol.}, 59(6):1335--1362, 2015.

\bibitem{GIESEKUS1982}
H.~Giesekus.
\newblock A simple constitutive equation for polymer fluids based on the
  concept of deformation-dependent tensorial mobility.
\newblock {\em Journal of Non-Newtonian Fluid Mechanics}, 11(1):69--109, 1982.

\bibitem{Hopf1950}
E.~Hopf.
\newblock Über die anfangswertaufgabe für die hydrodynamischen
  grundgleichungen. erhard schmidt zu seinem 75. geburtstag gewidmet.
\newblock {\em Mathematische Nachrichten}, 4(1-6):213--231, 1950.

\bibitem{JOHNSON1977}
M.~W. Johnson and D.~Segalman.
\newblock A model for viscoelastic fluid behavior which allows non-affine
  deformation.
\newblock {\em Journal of Non-Newtonian Fluid Mechanics}, 2(3):255--270, 1977.

\bibitem{Larson88}
R.~G Larson.
\newblock {\em Constitutive Equations for Polymer Melts and Solutions}.
\newblock Butterworths Series in Chemical Engineering. Butterworths, London,
  1988.

\bibitem{leray.j:sur}
J.~Leray.
\newblock Sur le mouvement d'un liquide visqueux emplissant l'espace.
\newblock {\em Acta Math.}, 63:193--248, 1934.

\bibitem{MalekPrusaChapter}
J.~M\'{a}lek and V.~Pr\accent23u\v{s}a.
\newblock Derivation of equations for continuum mechanics and thermodynamics of
  fluids.
\newblock In {\em Handbook of mathematical analysis in mechanics of viscous
  fluids}, pages 3--72. Springer, Cham, 2018.

\bibitem{MalekPrusaSkrivanSuli18}
J.~M\'alek, V.~Pr\accent23u\v{s}a, T.~Sk\v{r}ivan, and E.~S\"uli.
\newblock Thermodynamics of viscoelastic rate-type fluids with stress
  diffusion.
\newblock {\em Physics of Fluids}, 30(2):023101, 2018.

\bibitem{MRT2015}
J.~M\'alek, K.~R. Rajagopal, and K.~T\r{u}ma.
\newblock On a variant of the {M}axwell and {O}ldroyd-{B} models within the
  context of a thermodynamic basis.
\newblock {\em Int. J. Non-Linear Mech.}, 76:42--47, 2015.

\bibitem{Oldroyd}
J.~G. Oldroyd.
\newblock On the formulation of rheological equations of state.
\newblock {\em Proc. Roy. Soc. London Ser. A}, 200:523--541, 1950.

\bibitem{olmsted.pd:perspectives}
P.~D. Olmsted.
\newblock Perspectives on shear banding in complex fluids.
\newblock {\em Rheol. Acta}, 47(3):283--300, 2008.

\bibitem{OlmstedRadulescu2000}
P.~D. Olmsted, O.~Radulescu, and C.-Y.~D. Lu.
\newblock Johnson-{S}egalman model with a diffusion term in cylindrical
  {C}ouette flow.
\newblock {\em Journal of Rheology}, 44(2):257--275, 2000.

\bibitem{RAJAGOPAL2000}
K.~R. Rajagopal and A.~R. Srinivasa.
\newblock A thermodynamic frame work for rate type fluid models.
\newblock {\em Journal of Non-Newtonian Fluid Mechanics}, 88(3):207--227, 2000.

\bibitem{RajSri2004}
K.~R. Rajagopal and A.~R. Srinivasa.
\newblock On thermomechanical restrictions of continua.
\newblock {\em Proc. R. Soc. Lond. Ser. A Math. Phys. Eng. Sci.},
  460(2042):631--651, 2004.

\bibitem{Temam85}
R.~Temam.
\newblock {\em Navier-{S}tokes equations and nonlinear functional analysis},
  volume~66 of {\em CBMS-NSF Regional Conference Series in Applied
  Mathematics}.
\newblock Society for Industrial and Applied Mathematics (SIAM), Philadelphia,
  PA, second edition, 1995.

\bibitem{temam.r:navier-stokes}
R.~Temam.
\newblock {\em Navier-{S}tokes equations}.
\newblock AMS Chelsea Publishing, Providence, RI, 2001.
\newblock Theory and numerical analysis, Reprint of the 1984 edition.

\bibitem{NLFTM}
C.~Truesdell and W.~Noll.
\newblock {\em The non-linear field theories of mechanics}.
\newblock Springer-Verlag, Berlin, third edition, 2004.
\newblock Edited and with a preface by Stuart S. Antman.

\end{thebibliography}

\bigskip

\centerline{\scshape Miroslav Bul\'\i{}\v{c}ek}
\smallskip
{\footnotesize
 \centerline{Charles University, Faculty of Mathematics and Physics, Mathematical Institute,}
 \centerline{
 Sokolovsk\' a 83, 18675 Prague 8, Czech Republic}
\centerline{\email{mbul8060@karlin.mff.cuni.cz}}
}
\bigskip

\centerline{\scshape Josef M\'{a}lek}
\smallskip
{\footnotesize
 \centerline{Charles University, Faculty of Mathematics and Physics, Mathematical Institute,}
 \centerline{Sokolovsk\' a 83, 18675 Prague 8, Czech Republic}
\centerline{\email{mbul8060@karlin.mff.cuni.cz}}
}
\bigskip

\centerline{\scshape Casey Rodriguez}
\smallskip
{\footnotesize
 \centerline{Department of Mathematics, University of North Carolina,}
 \centerline{Chapel Hill, NC, USA}
\centerline{\email{crodrig@email.unc.edu}}
}

\end{document}